\documentclass[a4paper,11pt,reqno]{amsart}
\usepackage{amsmath}
\usepackage{amssymb}

\usepackage{mathrsfs}
\renewcommand{\mathcal}{\mathscr}

\theoremstyle{plain} 
\newtheorem{theorem}{Theorem}[section]

\newtheorem{lemma}[theorem]{Lemma}
\newtheorem{corollary}[theorem]{Corollary}
\newtheorem{remark}[theorem]{Remark}

\makeatletter
    
    \@addtoreset{equation}{section}
  \makeatother

\title[Superdiffusive elephant random walks]{Gaussian fluctuation for superdiffusive elephant random walks}
\author{Naoki Kubota}
\address{College of Science and Technology, Nihon University}
\email{kubota.naoki08@nihon-u.ac.jp}
\author{Masato Takei}
\address{Department of Applied Mathematics, Faculty of Engineering, Yokohama National University}
\email{takei-masato-fx@ynu.ac.jp}

\begin{document}
\maketitle

\begin{abstract}
Elephant random walk is a kind of one-dimensional discrete-time random walk with infinite memory: For each step, with probability $\alpha$ the walker adopts one of his/her previous steps uniformly chosen at random, and otherwise he/she performs like a simple random walk (possibly with bias). It admits phase transition from diffusive to superdiffusive behavior at the critical value $\alpha_c=1/2$. For $\alpha \in (\alpha_c, 1)$, there is a scaling factor $a_n$ of order $n^{\alpha}$ such that the position $S_n$ of the walker at time $n$ scaled by $a_n$ converges to a nondegenerate random variable $W$, whose distribution is not Gaussian. Our main result shows that the fluctuation of $S_n$ around $W \cdot a_n$ is still Gaussian. We also give a description of phase transition induced by bias decaying polynomially in time.
\end{abstract}

\section{Introduction}
\label{intro}

The elephant random walk, introduced by Sch\"{u}tz and Trimper \cite{SchutzTrimper04}, is one of the simplest models of step reinforced random walks: \begin{itemize}
\item The first step $X_1$ of the walker is $+1$ with probability $q$, and $-1$ with probability $1-q$.
\item For each $n=1,2,\cdots$, let $U_n$ be uniformly distributed on $\{1,\cdots,n\}$, and
\begin{align*}
X_{n+1} &=  \begin{cases}
X_{U_n} &\mbox{with probability $p$}, \\
-X_{U_n} &\mbox{with probability $1-p$}. \\
\end{cases}
\end{align*} 
\end{itemize}
Each of choices in the above procedure is made independently. The sequence $\{X_i\}$ generates a one-dimensional random walk $\{S_n\}$ by
\[ S_0:=0,\quad \mbox{and} \quad S_n= \sum_{i=1}^n X_i \quad \mbox{for $n=1,2,\cdots$.} \]

We begin with a brief review of the result in \cite{SchutzTrimper04}.
Let $\mathcal{F}_n$ be the $\sigma$-algebra generated by $X_1,\cdots,X_n$. 
For $n=1,2,\cdots$, the conditional distribution of $X_{n+1}$ given the history up to time $n$ is
\begin{align}
 &P(X_{n+1}= \pm 1 \mid \mathcal{F}_n) \notag \\
 &= \dfrac{\#\{i=1,\cdots,n : X_i=\pm 1\}}{n} \cdot p+  \dfrac{\#\{i=1,\cdots,n : X_i= \mp 1\}}{n} \cdot (1-p) \notag \\
 &= (2p-1) \cdot \dfrac{\#\{i=1,\cdots,n : X_i=\pm 1\}}{n} + 1-p, \label{eq:elephantRWCondDistp}
\end{align}
and the conditional expectation of $X_{n+1}$ is 
\[ E[X_{n+1} \mid \mathcal{F}_n]=(2p-1) \cdot \dfrac{S_n}{n}. \]
The new parameters defined by $\alpha:=2p-1$ and $\beta:=E[X_1]=2q-1$ will be convenient later. Noting that
\[ E[S_{n+1} \mid \mathcal{F}_n]=\left(1+\dfrac{\alpha}{n}\right) S_n, \]
we introduce 
\begin{align} \label{eq:ElephantRWa_nDef}
 a_0:=1, \quad \mbox{and} \quad 
 a_n := \prod_{k=1}^{n-1} \left(1+\dfrac{\alpha}{k}\right)= \dfrac{\Gamma(n+\alpha)}{\Gamma(n)\Gamma(\alpha+1)}\quad \mbox{for $n=1,2,\cdots$.}
\end{align}
Set
\begin{align*}
M_n:=\dfrac{S_n}{a_n}\quad \mbox{for $n=0,1,2,\cdots$.}
\end{align*}
Then $\{M_n\}$ has a martingale property $E[M_{n+1} \mid \mathcal{F}_n]=M_n$. In particular we have $E[M_n] = E[M_1]=E[X_1]=\beta$ and $ E[S_n]=\beta a_n$.
By the Stirling formula for Gamma functions, 
\begin{align}
a_n \sim \dfrac{n^{\alpha}}{\Gamma(\alpha+1)}\quad\mbox{as $n \to \infty$}, \label{eq:ElephantRWanAsymp}
\end{align}
where $x_n \sim y_n$ means that $x_n/y_n$ converges to $1$ as $n \to \infty$. 
By a further calculation, we can see that the mean square displacement satisfies
\[
E[(S_n)^2] 
\sim 
\begin{cases}
\dfrac{1}{1-2\alpha}n &(\alpha<1/2), \\
n\log n &(\alpha=1/2), \\
\dfrac{1}{(2\alpha-1)\Gamma(2\alpha)} n^{2\alpha} &(\alpha>1/2).
\end{cases}
\]

When $\alpha < 1/2$, the elephant random walk is diffusive, and the fluctuation is Gaussian: $S_n/\sqrt{\frac{n}{1-2\alpha}} \stackrel{\text{d}}{\to} N(0,1)$, 
where $\stackrel{\text{d}}{\to}$ denotes the convergence in distribution as $n \to \infty$, and $N(0,1)$ is the standard normal distribution. When $\alpha = 1/2$, the walk is marginally superdiffusive, but still $S_n/\sqrt{n \log n}\stackrel{\text{d}}{\to} N(0,1)$. On the other hand, if $\alpha>1/2$, then $\{M_n\}$ is an $L^2$-bounded martingale, and the martingale convergence theorem shows that $S_n/n^{\alpha}$ converges to a non-degenerate random variable with mean $\tfrac{\beta}{\Gamma(\alpha+1)}$, whose distribution turns out to be non-Gaussian (see \cite{Bercu18,BercuChabanolRuch19} among others), a.s. and in $L^2$. These and further strong limit theorems are obtained by \cite{BaurBertoin16,Bercu18,Collettietal17a,Collettietal17b}. K\"{u}rsten \cite{Kursten16} relates the phase transition described above to the behavior of a spin system on random recursive trees. Variations of elephant random walks studied mainly from  mathematical viewpoint are found in \cite{BercuLaulin19,Bertoin18,Businger18,GutStadtmuller18,GutStadtmuller19}.

By the way, \eqref{eq:elephantRWCondDistp} is equivalent to
\begin{align}
P(X_{n+1}= \pm 1 \mid \mathcal{F}_n) = \alpha \cdot \dfrac{\#\{i=1,\cdots,n : X_i=\pm 1\}}{n} + (1-\alpha) \cdot \dfrac{1}{2}.
\label{eq:elephantRWCondDistalpha}
\end{align}
If $\alpha \in [0,1]$ (i.e. $p \in [1/2,1]$), then we have the following interpretation:
\begin{itemize}
\item With probability $\alpha$, the walker repeats one of his/her previous steps.
\item With probability $1-\alpha$, the walker performs like a simple symmetric random walk.
\end{itemize}
Drezner and Farnum \cite{DreznerFarnum93} studied a closely related problem: In our notation, their model is defined by setting $\beta=\varepsilon$ and
\begin{align}
P(X_{n+1}= \pm 1 \mid \mathcal{F}_n) = \alpha_n \cdot \dfrac{\#\{i=1,\cdots,n : X_i=\pm 1\}}{n} + (1-\alpha_n) \cdot \dfrac{1\pm \varepsilon }{2},
\end{align}
where $\alpha_n \in [0,1]$ and $\varepsilon \in [-1,1]$. In \cite{DreznerFarnum93}, $\{X_i\}$ is called a {\it correlated Bernoulli process} and the distribution of $H_n:=\#\{i=1,\cdots,n : X_i=+ 1\}$ the {\it generalized binomial distribution} with density $\rho:=\dfrac{1+ \varepsilon }{2}$. In this context, various limit theorems for $\{H_n\}$ are obtained by \cite{Heyde04,Jamesetal08,WuQiYang12SPL,ZhangZhang15SPL}. 
Note that those results have a counterpart for the elephant random walk by a simple relation $S_n=H_n-(n-H_n)=2H_n-n$.

In this paper we consider a reasonably wide class of elephant-type step-reinforced random walks, and investigate their limiting behavior. The rest is organized as follows. In section \ref{sec:Results} we give a precise definition of our model and statements of results. The main result in this paper is Theorem \ref{thm:KubotaTakeiSupercriticalCLTLIL}, which says that even in the supercritical regime the fluctuation of the position from the {\it random drift induced by memory effect} is Gaussian. This and related limit theorems (Theorems \ref{thm:Bercu18sLLN} and \ref{thm:Jamesetal08modelMainSubcrit}) are proved in section \ref{sec:ProofLimitThm}. Theorem \ref{thm:Bercu18sLLN} below shows that step-reinforcement does not change the asymptotic speed of {\it asymmetric} simple random walks. We study the effect of step-reinforcement for {\it asymptotically symmetric} simple random walks in Theorem \ref{thm:ElephantPolynomDecay}, which will be proved in section \ref{sec:ProofERWbias}. The facts on calculus and martingale limit theorems on which we rely are summarized in the appendix.

\section{Results} \label{sec:Results}

Hereafter we consider the following class of elephant random walks, namely one-dimensional nearest-neighbor random walks, whose bias can depend on time, with step-reinforcement:
\begin{align}
 P(X_{n+1}= \pm 1 \mid \mathcal{F}_n) 
 &= \alpha \cdot \dfrac{\#\{i=1,\cdots,n : X_i=\pm 1\}}{n} + (1-\alpha) \cdot \dfrac{1 \pm \varepsilon_n}{2}.  \label{eq:Jamesetal08modelDef} 
\end{align}

The following theorems are generalizations of the results obtained in \cite{BaurBertoin16,Bercu18,Collettietal17a,Collettietal17b} for the original elephant random walk, although they are essentially proved in existing literatures (\cite{Jamesetal08,WuQiYang12SPL} among others) for the correlated Bernoulli process. To make this paper reasonably self-contained, we indicate the main lines of proofs in section \ref{sec:ProofLimitThm}.

\begin{theorem} \label{thm:Bercu18sLLN} Assume \eqref{eq:Jamesetal08modelDef}. For any $\alpha \in [0,1)$,
\begin{equation}
\displaystyle \lim_{n \to \infty} \dfrac{S_n-E[S_n]}{n} = 0 \quad \mbox{a.s..} \label{eq:Bercu18sLLNa}
\end{equation}
If $\displaystyle \lim_{n \to \infty} \varepsilon_n =\varepsilon \in [0,1)$ in addition, then
\begin{equation}
\displaystyle \lim_{n \to \infty} \dfrac{E[S_n]}{n} =\varepsilon \quad \mbox{and} \quad  \lim_{n \to \infty} \dfrac{S_n}{n}= \varepsilon \quad \mbox{a.s..} \label{eq:Bercu18sLLNb}
\end{equation}
\end{theorem}


\begin{theorem} \label{thm:Jamesetal08modelMainSubcrit} Assume that \eqref{eq:Jamesetal08modelDef} and $\displaystyle \lim_{n \to \infty} \varepsilon_n =\varepsilon \in [0,1)$ hold. Let $\phi(t):=\sqrt{2t\log\log t}$.
\begin{itemize}
\item[(i)] If $0 \leq \alpha <1/2$, then
\begin{align*}
\dfrac{S_n-E[S_n]}{\sqrt{\frac{1-\varepsilon^2}{1-2\alpha} n}} \stackrel{\text{d}}{\to} N\left(0,1\right),\mbox{ and } 
\limsup_{n \to \infty} \pm \dfrac{S_n-E[S_n]}{\phi\left(\frac{1-\varepsilon^2}{1-2\alpha} n\right)}=1 
\mbox{ a.s..}
\end{align*}
\item[(ii)] If $\alpha=1/2$, then
\begin{align*}
\dfrac{S_n-E[S_n]}{\sqrt{(1-\varepsilon^2)n\log n}} \stackrel{\text{d}}{\to} N\left(0,1\right), \mbox{ and } 
\limsup_{n \to \infty} \pm \dfrac{S_n-E[S_n]}{\phi\left((1-\varepsilon^2)n\log n\right)}=1
\mbox{ a.s..}
\end{align*}
\item[(iii)] If $1/2<\alpha<1$, then there exists a random variable $W$ with positive variance such that
\[
\displaystyle \lim_{n \to \infty} \dfrac{S_n-E[S_n]}{a_n} = W \quad \mbox{a.s. and in $L^2$.}
\]
Since $P(W=0)<1$, if $\varepsilon=\beta=0$, then 
\[ P(W>0)=P(W<0)>0. \]
\end{itemize}
\end{theorem}

\subsection{Gaussian fluctuation for superdiffusive phase}

For the supercritical case, $W \cdot a_n$ can be regarded as the random drift induced by memory effects --- this point of view seems to escape from attention in previous studies. Our main result shows that the fluctuation of elephant random walk from the random drift is still Gaussian, and is striking particularly for the case $E[S_n] = o(n^{\alpha})$ as $n \to \infty$, which is equivalent to $\varepsilon=0$ (e.g. the original elephant random walk) by Theorem \ref{thm:Bercu18sLLN}. This type of result is apparently new even for the correlated Bernoulli process.

\begin{theorem} \label{thm:KubotaTakeiSupercriticalCLTLIL} 
Under the condition of Theorem \ref{thm:Jamesetal08modelMainSubcrit} (iii), we have
\begin{align*}
\dfrac{S_n-E[S_n]-W \cdot a_n}{\sqrt{\tfrac{1-\varepsilon^2}{2\alpha-1} n}} \stackrel{\text{d}}{\to} N\left(0,1\right),
\intertext{ and }
\limsup_{n \to \infty} \pm \dfrac{S_n-E[S_n]-W \cdot a_n}{a_n \cdot \widehat{\phi}\left( \tfrac{1-\varepsilon^2}{2\alpha-1} n\right)}
=1 
\mbox{ a.s.,}
\end{align*}
where $\widehat{\phi}(t):=\sqrt{2t\log |\log t|}$. 
\end{theorem}
\begin{remark} When $\alpha=1$, since $\displaystyle \dfrac{S_n}{n} = X_1$ for any $n=1,2,\cdots$, the central limit theorem for $\{S_n-X_1 \cdot n\}$ does not hold.
\end{remark}

\subsection{Elephant random walk with polynomially decaying bias}

By Theorem \ref{thm:Bercu18sLLN}, if $\varepsilon_n \to \varepsilon \in (0,1)$ as $n \to \infty$, then the asymptotic speed of the walker is not affected by $\alpha$ at all.
Thus we are interested in the case $\varepsilon_n$ vanishes as $n \to \infty$. 
The following theorem describes various phases arising for the long time behavior of $\{S_n\}$ in this setting, and shows that $\gamma=1/2$ is critical. The convergence in $L^2$ as $n \to \infty$ is denoted by $\stackrel{L^2}{\to}$.

\begin{theorem} \label{thm:ElephantPolynomDecay} Assume that $\varepsilon_n = n^{-\gamma}$ with $\gamma>0$. The limiting distribution of the deviation $\{ S_n-E[S_n] \}$ from the mean is given by Theorem \ref{thm:Jamesetal08modelMainSubcrit} with $\varepsilon=0$. In addition, we have the following for $\{S_n\}$.
\begin{itemize}
\item[(i)] Suppose that $0<\gamma<1/2$.
\begin{itemize}
\item[a)] If $0 \leq \alpha < 1-\gamma$, then $\dfrac{S_n}{n^{1-\gamma}} \stackrel{L^2}{\to} \dfrac{1-\alpha}{1-\gamma-\alpha} >0$.
\item[b)] If $\alpha=1-\gamma$, then $\dfrac{S_n}{n^{\alpha}\log n} \stackrel{L^2}{\to} 1-\alpha >0$. 
\item[c)] If $1-\gamma<\alpha<1$, then $\displaystyle \lim_{n \to \infty} \dfrac{S_n}{n^{\alpha}} =\widehat{W}$ a.s. and in $L^2$.
\end{itemize}
\item[(ii)] Suppose that $\gamma = 1/2$.
\begin{itemize}
\item[a)] If $0 \leq \alpha < 1/2$, then 
$\dfrac{S_n}{\sqrt{n}} \stackrel{\text{d}}{\to} N\left(\dfrac{2-2\alpha}{1-2\alpha},\dfrac{1}{1-2\alpha}\right)$.
\item[b)] If $\alpha=1/2$, then $\dfrac{S_n}{\sqrt{n (\log n)^2}} \stackrel{L^2}{\to}\dfrac{1}{2} $. 
\item[c)] If $1/2<\alpha<1$, then $\displaystyle \lim_{n \to \infty} \dfrac{S_n}{n^{\alpha}} =\widehat{W}$ a.s. and in $L^2$.
\end{itemize}
\item[(iii)] Suppose that $\gamma>1/2$. 
\begin{itemize}
\item[a)] If $0 \leq \alpha < 1/2$, then $\dfrac{S_n}{\sqrt{\frac{n}{1-2\alpha}}} \stackrel{\text{d}}{\to} N\left(0,1\right)$, and 
$\displaystyle \limsup_{n \to \infty} \pm \dfrac{S_n}{\phi\left(\frac{n}{1-2\alpha} \right)}=1$ a.s..
\item[b)] If $\alpha = 1/2$, then
$\dfrac{S_n}{\sqrt{n\log n}} \stackrel{\text{d}}{\to} N\left(0,1\right)$, and $\displaystyle \limsup_{n \to \infty} \pm \dfrac{S_n}{\phi\left(n\log n\right)}=1$ a.s..
\item[c)] If $1/2<\alpha<1$, then $\displaystyle \lim_{n \to \infty} \dfrac{S_n}{n^{\alpha}} =\widehat{W}$ a.s. and in $L^2$.
\end{itemize}
\end{itemize}
The random variable $\widehat{W}$ satisfies that  $E[\widehat{W}]>\tfrac{\beta}{\Gamma(\alpha+1)}$ and $E[\widehat{W}^2] \in (0,+\infty)$.
\end{theorem}

In the case (i), there is a competition between the memory effect and the bias,
and the walker is (zero-speed) transient. Together with Theorems \ref{thm:Jamesetal08modelMainSubcrit} and \ref{thm:KubotaTakeiSupercriticalCLTLIL}, for a fixed $\gamma < 1/2$ we can observe several transition of limiting behavior as $\alpha$ increases from $0$ to $1$. On the other hand, in the case (iii) the bias vanishes `rapidly' and the limit theorems given above is qualitatively the same as the original elephant random walk ($\varepsilon_n \equiv 0$). Somewhat peculiar behavior is found in the `critical' case (ii).

\section{Limit theorems} \label{sec:ProofLimitThm}

Let $\alpha \in [0,1)$ and $\{\varepsilon_n\}_{n=1,2,\cdots} \subset [0,1]$.
Assume that the conditional distribution of $X_{n+1}$ is given by \eqref{eq:Jamesetal08modelDef}.
For $n=1,2,\cdots$, we have
\begin{align}
 E[X_{n+1} \mid \mathcal{F}_n]&=\alpha \cdot \dfrac{S_n}{n} + (1-\alpha) \cdot \varepsilon_n,\label{eq:Jamesetal08modelDef2}
\intertext{and}
 E[S_{n+1} \mid \mathcal{F}_n]&=\gamma_n S_n+ (1-\alpha) \cdot \varepsilon_n, \label{eq:Jamesetal08modelDef3}
\end{align}
where $\gamma_n:=1+\dfrac{\alpha}{n}$. For $n=0,1,2,\cdots$, we set
\[ M_n:=\dfrac{S_n-E[S_n]}{a_n}, \]
where $a_n$ is defined in \eqref{eq:ElephantRWa_nDef}. Let $\mathcal{F}_0$ denote the trivial $\sigma$-algebra. 

\begin{lemma} \label{lem:ElephantRWmartdiffn} 
The sequence $\{M_n\}$ is a square-integrable martingale with mean $0$.
\end{lemma}

\begin{proof} For each $n=1,2,\cdots$, we have
\begin{align}
M_{n+1}-M_n &= \dfrac{S_{n+1}-E[S_{n+1}]}{a_{n+1}}-\dfrac{S_n-E[S_n]}{a_n} \notag \\
&= \dfrac{S_{n+1}-\gamma_n S_n-E[S_{n+1}-\gamma_n S_n] }{a_{n+1}} \notag \\ 
&= \dfrac{S_{n+1}-\gamma_n S_n-(1-\alpha)\varepsilon_n}{a_{n+1}} \notag \\
&= \dfrac{S_{n+1}-E[S_{n+1} \mid \mathcal{F}_n]}{a_{n+1}} = \dfrac{X_{n+1}-E[X_{n+1} \mid \mathcal{F}_n]}{a_{n+1}},\label{eq:ElephantRWmartdiff1}
\end{align}
which shows that $E[M_{n+1}-M_n \mid \mathcal{F}_n]=0$. Moreover,
\begin{align}
E[M_{n+1}^2 - M_n^2 \mid \mathcal{F}_n] &= E[(M_{n+1}-M_n)^2 \mid \mathcal{F}_n] \notag \\
&=\dfrac{E[(X_{n+1}-E[X_{n+1} \mid \mathcal{F}_n])^2 \mid \mathcal{F}_n]}{(a_{n+1})^2} \notag \\
&=\dfrac{E[X_{n+1}^2 \mid \mathcal{F}_n] -(E[X_{n+1} \mid \mathcal{F}_n])^2}{(a_{n+1})^2} \notag \\
&=\dfrac{1 -(E[X_{n+1} \mid \mathcal{F}_n])^2}{(a_{n+1})^2}.\label{eq:ElephantRWmartdiff2}
\end{align}
Note that \eqref{eq:ElephantRWmartdiff1} and \eqref{eq:ElephantRWmartdiff2} hold also for $n=0$. Since $|X_{n+1}|=1$, we can see 
that $E[(M_n)^2]<+\infty$ for each $n$.
\end{proof}

For $k=1,2,\cdots$, let
\[ d_k:=M_k-M_{k-1}=\dfrac{X_k-E[X_k \mid \mathcal{F}_{k-1}]}{a_k}. \]
Note that $|d_k| \leq \dfrac{2}{a_k} (\leq 2)$.

\begin{lemma} \label{lem:BERW171109} Assume \eqref{eq:Jamesetal08modelDef}. For $n=1,2,\cdots$,
\begin{equation}
E[S_n] = \beta a_n + (1-\alpha) a_n \sum_{\ell=1}^{n-1} \dfrac{\varepsilon_{\ell}}{a_{\ell+1}}. \label{eq:BERW171109E}
\end{equation}
\end{lemma}

\begin{proof} Solving the recursion
\[ E[S_1]=\beta \quad \mbox{and} \quad E[S_{n+1}] = \gamma_n E[S_n] + (1-\alpha) \varepsilon_n \]
by Lemma \ref{SchutzTrimper04Recursion}, we have
\begin{align*}
E[S_n] = \beta \prod_{k=1}^{n-1} \gamma_k + (1-\alpha) \sum_{\ell=1}^{n-1} \varepsilon_{\ell} \prod_{k=\ell+1}^{n-1} \gamma_k
= \beta a_n + (1-\alpha) \sum_{\ell=1}^{n-1} \varepsilon_{\ell} \cdot \dfrac{a_n}{a_{\ell+1}}.
\end{align*}
\end{proof}


\begin{proof}[Proof of Theorem \ref{thm:Bercu18sLLN}] The proof of \eqref{eq:Bercu18sLLNa}, where the strong law of large numbers for martingales (Theorem \ref{thm:HallHeyde80Thm215}) and Kronecker's lemma (Lemma \ref{lem:KroneckerHeyde}) are used, is almost the same as Theorem 1 in \cite{Collettietal17a}, and is omitted.
We turn to \eqref{eq:Bercu18sLLNb}. 
Since $\alpha<1$, \eqref{eq:ElephantRWanAsymp} implies that the first term in the right hand side of \eqref{eq:BERW171109E} is $o(n)$ as $n \to \infty$. Now we rewrite the second term in the right hand side of \eqref{eq:BERW171109E} as 
\begin{align*}
(1-\alpha) a_n \left(\sum_{\ell=1}^{n-1} \dfrac{1}{a_{\ell+1}}\right) \cdot \dfrac{\sum_{\ell=1}^{n-1}\frac{\varepsilon_{\ell}}{a_{\ell+1}}}{\sum_{\ell=1}^{n-1} \frac{1}{a_{\ell+1}}} .
\end{align*}
Noting that 
\begin{align*}
(1-\alpha) a_n \left(\sum_{\ell=1}^{n-1} \dfrac{1}{a_{\ell+1}}\right) \sim (1-\alpha) \cdot \dfrac{n^{\alpha}}{\Gamma(\alpha+1)} \cdot \Gamma(\alpha+1) \cdot \dfrac{n^{1-\alpha}}{1-\alpha} = n
\end{align*}
as $n \to \infty$, and that $\displaystyle \sum_{\ell=1}^{\infty} \dfrac{1}{a_{\ell+1}}=+\infty$ since $\alpha<1$. If $\displaystyle \lim_{n \to \infty} \varepsilon_n =\varepsilon$，then Lemma \ref{lem:StolzCesaro} implies that
\begin{align*}
\lim_{n \to \infty} \dfrac{\sum_{\ell=1}^{n-1}\frac{\varepsilon_{\ell}}{a_{\ell+1}}}{\sum_{\ell=1}^{n-1} \frac{1}{a_{\ell+1}}} = \varepsilon.
\end{align*}
Thus we have $\displaystyle \lim_{n \to \infty} \dfrac{E[S_n]}{n} =\varepsilon$.
\end{proof}

\begin{proof}[Proof of Theorem \ref{thm:Jamesetal08modelMainSubcrit}] By Theorem \ref{thm:Bercu18sLLN} and \eqref{eq:Jamesetal08modelDef2}, we have
\begin{align*}
E[X_{n+1} \mid \mathcal{F}_n] = \alpha \cdot \dfrac{S_n}{n} + (1-\alpha) \varepsilon_n \to \alpha \varepsilon + (1-\alpha) \varepsilon =  \varepsilon \quad \mbox{as $n \to \infty$.}
\end{align*}
This together with \eqref{eq:ElephantRWmartdiff2} implies that
\begin{align}
E\left[ (d_k)^2 \mid \mathcal{F}_{k-1} \right] \sim \dfrac{1-\varepsilon^2}{(a_k)^2} \quad \mbox{as $k \to \infty$}. \label{eq:Jamesetal08modelMainSubcrit1}
\end{align}
By \eqref{eq:ElephantRWanAsymp}, if $\alpha \leq 1/2$, then
\begin{align*}
 \sum_{k=1}^n E\left[ (d_k)^2 \mid \mathcal{F}_{k-1} \right] \sim (1-\varepsilon^2)\Gamma(\alpha+1)^2 \sum_{k=1}^n \dfrac{1}{k^{2\alpha}} 
\end{align*}
as $n \to \infty$. The right hand side is
\begin{align*}
\sim \begin{cases}
(1-\varepsilon^2)\Gamma(\alpha+1)^2 \cdot \dfrac{n^{1-2\alpha}}{1-2\alpha} \sim \dfrac{1-\varepsilon^2}{1-2\alpha}n \cdot \dfrac{1}{(a_n)^2}&(\alpha<1/2), \\
(1-\varepsilon^2)\Gamma\left(\dfrac{3}{2} \right)^2\cdot \log n \sim(1-\varepsilon^2) n \log n\cdot \dfrac{1}{(a_n)^2}&(\alpha=1/2) \\
\end{cases}
\end{align*}
as $n \to \infty$.
Thus Theorem \ref{thm:Jamesetal08modelMainSubcrit} (i) and (ii) follow from Lemmas 3.4 and 3.5 in \cite{Jamesetal08}.
Now we consider the case $1/2<\alpha<1$.
By \eqref{eq:Jamesetal08modelMainSubcrit1} and the bounded convergence theorem, we have
\[
E[(d_k)^2] \sim \dfrac{1-\varepsilon^2}{(a_k)^2}\quad \mbox{as $n \to \infty$.}
\]
In view of \eqref{eq:ElephantRWanAsymp}, we have $\displaystyle \sum_{k=1}^{\infty} E[(d_k)^2] < +\infty$ when $1/2<\alpha<1$. Theorem \ref{thm:Heyde77Theorem1b} (i) implies that
\[
W := \sum_{k=1}^{\infty}d_k=\lim_{n \to \infty} M_n =  \lim_{n \to \infty} \dfrac{S_n-E[S_n]}{a_n}  
\]
exists with probability one, and since $M_n \stackrel{L^2}{\to} W$, we have
\[ E[W]=0,\quad \mbox{and} \quad E[W^2] = \sum_{k=1}^{\infty} E[(d_k)^2] > 0. \]
This completes the proof.
\end{proof}

\begin{proof}[Proof of Theorem \ref{thm:KubotaTakeiSupercriticalCLTLIL}] 
We check the conditions of Theorem \ref{thm:Heyde77Theorem1b} (ii) and (iii) are satisfied.
When $\alpha > 1/2$,
\begin{align*}
 V_n^2&:= \sum_{k=n}^{\infty} E\left[ (d_k)^2 \mid \mathcal{F}_{k-1} \right] \\
 &\sim (1-\varepsilon^2)\Gamma(\alpha+1)^2 \sum_{k=n}^{\infty} \dfrac{1}{k^{2\alpha}}=\dfrac{(1-\varepsilon^2)\Gamma(\alpha+1)^2}{2\alpha-1} \cdot \dfrac{1}{n^{2\alpha-1}} \\
 &\sim \dfrac{1-\varepsilon^2}{2\alpha-1}n \cdot \dfrac{1}{(a_n)^2} \quad \mbox{a.s.,} 
\end{align*}
and
\begin{align*}
 s_n^2&:=\sum_{k=n}^{\infty} E[(d_k)^2] \sim \dfrac{1-\varepsilon^2}{2\alpha-1}n \cdot \dfrac{1}{(a_n)^2}
\end{align*}
as $n  \to \infty$. 
This implies that 
\begin{align}
\lim_{n \to \infty} \dfrac{V_n^2}{s_n^2}=1 \quad \mbox{a.s..} \label{eq:KubotaTakeiSupercriticalCLTLILV_nLLN}
\end{align}
Noting that
\[ s_n^4 \sim  \dfrac{\Gamma(\alpha+1)^4}{(2\alpha -1)^2} \cdot \dfrac{(1-\varepsilon^2)^2}{n^{4\alpha-2}},  \]
and
\begin{align*}
(d_n)^4 
\leq \dfrac{16}{(a_n)^4}
\sim \dfrac{16\Gamma(\alpha+1)^4}{n^{4\alpha}}
\end{align*}
as $n  \to \infty$, we obtain
\[ 
\sum_{n=1}^{\infty} \dfrac{1}{s_n^4} E[(d_n)^4 \mid \mathcal{F}_{n-1}] \leq C_1 \sum_{n=1}^{\infty} \dfrac{1}{n^2}<+\infty.
\]
Theorem \ref{thm:HallHeyde80Thm215} shows that
\[
\sum_{k=1}^{\infty} \dfrac{1}{s_k^2} \{ (d_k)^2 - E[(d_k)^2 \mid \mathcal{F}_{k-1}]\}<+\infty \quad \mbox{a.s..}
\]
By Lemma \ref{lem:KroneckerHeyde} (ii),
\[
\lim_{n \to \infty} \dfrac{1}{s_n^2}\sum_{k=n}^{\infty} \{ (d_k)^2 - E[(d_k)^2 \mid \mathcal{F}_{k-1}]\} = 0 \quad \mbox{a.s..}
\]
This together with \eqref{eq:KubotaTakeiSupercriticalCLTLILV_nLLN} shows that conditions a) and a') are satisfied. 
For $\varepsilon>0$, since
\[
E[(d_k)^2 : |d_k| > \varepsilon s_n] \leq \dfrac{1}{\varepsilon^2 s_n^2} E[(d_k)^4],
\]
we have
\begin{align*}
\dfrac{1}{s_n^2} \sum_{k=n}^{\infty} E[(d_k)^2 : |d_k| > \varepsilon s_n] &\leq \dfrac{1}{\varepsilon^2 s_n^4} \sum_{k=n}^{\infty} E[(d_k)^4] \\
&\leq C_2 n^{4\alpha-2} \cdot n^{1-4\alpha} = \dfrac{C_2}{n} \to 0\quad \mbox{as $n \to \infty$}.
\end{align*}
In view of Remark \ref{rem:Heyde77Theorem1b}, condition b) is also satisfied. Similarly, for $\varepsilon>0$ we have
\begin{align*}
\dfrac{1}{s_k}E[|d_k| : |d_k| > \varepsilon s_k] &\leq \dfrac{1}{s_k} \cdot \dfrac{1}{\varepsilon^3 s_k^3} E[(d_k)^4] \\
&\leq C_3 k^{4\alpha-2} \cdot k^{-4\alpha} = \dfrac{C_3}{k^2},
\end{align*}
which implies that condition c) holds. Condition d) is implied by
\[
\displaystyle \sum_{n=1}^{\infty} \dfrac{1}{s_n^4} E[(d_n)^4]<+\infty.
\]
The desired conclusion follows from
\[ W-M_n = \dfrac{a_n \cdot W - (S_n-E[S_n])}{a_n}.  \]
\end{proof}

\section{Elephant random walk with polynomially decaying bias} \label{sec:ProofERWbias}

In this section we assume that $\varepsilon_n =n^{-\gamma}$ with $\gamma > 0$.
Since
\[ \dfrac{\varepsilon_{\ell}}{a_{\ell+1}} \sim \Gamma(\alpha + 1) \ell^{-(\gamma+\alpha)} \quad \mbox{as $\ell \to \infty$}, \]
the critical line for the asymptotic behavior of $E[S_n]$ is $\gamma+\alpha=1$:

\begin{lemma} \label{lem:171113biasedRWESn} Assume that $\varepsilon_n =n^{-\gamma}$ with $\gamma > 0$. As $n \to \infty$,
\begin{align*}
E[S_n] \sim \begin{cases}
C(\alpha,\beta,\gamma) \cdot n^{\alpha} &(\gamma>1-\alpha), \\
(1-\alpha)n^{\alpha} \log n = \gamma n^{1-\gamma} \log n &(\gamma=1-\alpha), \\
\dfrac{1-\alpha}{1-\gamma-\alpha} n^{1-\gamma} &(\gamma<1-\alpha),
\end{cases}
\end{align*}
where $C(\alpha,\beta,\gamma)$ is a constant larger than $\frac{\beta}{\Gamma(\alpha+1)}$.
\end{lemma}

\begin{proof} When $\gamma+\alpha>1$, noting that $\displaystyle \sum_{\ell=1}^{\infty} \dfrac{\varepsilon_{\ell}}{a_{\ell+1}}<+\infty$, Lemma \ref{lem:BERW171109} implies that
\begin{align*}
E[S_n] 
&\sim \dfrac{\beta  + (1-\alpha) \sum_{\ell=1}^{\infty} \frac{\varepsilon_{\ell}}{a_{\ell+1}}}{\Gamma(\alpha+1)} \cdot n^{\alpha} =:  C(\alpha,\beta,\gamma) \cdot n^{\alpha} 
\end{align*}
as $n \to \infty$. On the other hand, if $\gamma+\alpha \leq 1$, then
the second term in the right hand side of the equation \eqref{eq:BERW171109E} in Lemma \ref{lem:BERW171109} is dominant as $n \to \infty$, and we have
\begin{align*}
E[S_n] &\sim (1-\alpha)a_n  \sum_{\ell=1}^{n-1} \dfrac{\varepsilon_{\ell}}{a_{\ell+1}} \\
&\sim \dfrac{1-\alpha}{\Gamma(\alpha+1)}n^{\alpha} \cdot  \begin{cases}
\Gamma(\alpha+1) \log n &(\gamma+\alpha = 1), \\
\dfrac{\Gamma(\alpha+1)}{1-(\gamma+\alpha)} n^{1-(\gamma+\alpha)} &(\gamma+\alpha<1). \\
\end{cases} 
\end{align*}
This completes the proof. 
\end{proof}

Now we consider the cases where the effect of bias is weaker. 
Suppose that $\gamma>1/2$ and $\alpha \leq 1/2$.
By Lemma \ref{lem:171113biasedRWESn}, we can see that
\begin{align*}
E[S_n]=\begin{cases}
o(\sqrt{n}) &(\alpha<1/2), \\
o(\sqrt{n\log n}) &(\alpha=1/2) \\
\end{cases}
\quad \mbox{as $n \to \infty$.}
\end{align*}
This together with Theorem \ref{thm:Jamesetal08modelMainSubcrit} (i) and (ii) gives Theorem \ref{thm:ElephantPolynomDecay} (i) a) and b).
When $\alpha>1/2$ and $\gamma+\alpha>1$, Theorem \ref{thm:Jamesetal08modelMainSubcrit} (iii) and Lemma \ref{lem:171113biasedRWESn} shows that $S_n/n^{\alpha} \to \widehat{W}$ a.s. and in $L^2$, where $E[\widehat{W}]=C(\alpha,\beta,\gamma)$.
This implies part c) of Theorem \ref{thm:Jamesetal08modelMainSubcrit} (i), (ii) and (iii).
As for the case $\alpha < 1/2$ and $\gamma = 1/2$, Lemma \ref{lem:171113biasedRWESn} implies that
\begin{align*}
E[S_n] \sim  \dfrac{1-\alpha}{1-\tfrac{1}{2}-\alpha} n^{1/2} = \dfrac{2-2\alpha}{1-2\alpha} \sqrt{n} \quad \mbox{as $n \to \infty$},
\end{align*}
and Theorem \ref{thm:ElephantPolynomDecay} (ii) a) follows from Theorem \ref{thm:Jamesetal08modelMainSubcrit} (i).

Next we analyze the asymptotic behavior of $E[(S_n)^2]$ to prove Theorem \ref{thm:ElephantPolynomDecay} (i) a) and b).
By \eqref{eq:Jamesetal08modelDef}, we obtain
\begin{align*}
E[(S_{n+1})^2 \mid \mathcal{F}_n] 
&=(S_n)^2+2S_n \cdot E[X_{n+1} \mid \mathcal{F}_n]+ E[(X_{n+1})^2\mid \mathcal{F}_n] \\
&= \gamma'_n(S_n)^2+2(1-\alpha) \varepsilon_n S_n + 1,
\end{align*}
and
\[ E[(S_{n+1})^2] = \gamma'_n E[(S_n)^2]+2(1-\alpha) \varepsilon_n  E[S_n] + 1, \]
where $\gamma'_n:=1+\dfrac{2\alpha}{n}$. Noting that $E[(S_1)^2]=1$, Lemma \ref{SchutzTrimper04Recursion} implies that
\begin{align}
 E[(S_n)^2] &=  \sum_{\ell=0}^{n-1} \{ 2(1-\alpha) \varepsilon_\ell  E[S_\ell] + 1 \} \cdot \prod_{k=\ell+1}^{n-1} \gamma'_k \notag \\
 &= a'_n \sum_{\ell=0}^{n-1} \dfrac{1}{a'_{\ell+1}} + 2(1-\alpha) a'_n \sum_{\ell=0}^{n-1} \dfrac{\varepsilon_\ell  E[S_\ell]}{a'_{\ell+1}} .\label{eq:biasdecayESn2}
\end{align}
The first term in the right hand side is the second moment of $S_n$ for the elephant random walk with $\varepsilon_n \equiv 0$.
\begin{lemma} \label{lem:190306biasdecayESn2-3} If $\alpha \leq 1/2$ and $\gamma < 1/2$, or if $\alpha>1/2$ and $\gamma < 1-\alpha$，then $E[(S_n)^2] \sim (E[S_n])^2$ as $n \to \infty$.
\end{lemma}

\begin{proof} Since $\gamma+\alpha<1$ in both cases, 
Lemma \ref{lem:171113biasedRWESn} shows that
\begin{align*}
\dfrac{\varepsilon_{\ell}  E[S_\ell]}{a'_{\ell+1}} 
&\sim \ell^{-\gamma} \cdot  \dfrac{1-\alpha}{1-\gamma-\alpha} \ell^{1-\gamma}  \left/ \dfrac{\ell^{2\alpha}}{\Gamma(2\alpha+1)} \right. \\
&= \dfrac{1-\alpha}{1-\gamma-\alpha} \cdot \Gamma(2\alpha+1) \ell^{1-2(\gamma+\alpha)}\quad\mbox{as $\ell \to \infty$.}
\end{align*}
Since the second term in the right hand side of \eqref{eq:biasdecayESn2} is dominant as $n \to \infty$, we have
\begin{align*}
 E[(S_n)^2] &\sim 2(1-\alpha) a'_n \sum_{\ell=0}^{n-1} \dfrac{\varepsilon_\ell  E[S_\ell]}{a'_{\ell+1}} \\
 &\sim 2(1-\alpha) \cdot \dfrac{n^{2\alpha}}{\Gamma(2\alpha+1)} \cdot  \dfrac{1-\alpha}{1-\gamma-\alpha} \cdot \dfrac{\Gamma(2\alpha+1)}{2(1-\gamma-\alpha)} \cdot n^{2(1-\gamma-\alpha)} \\
 &= \left( \dfrac{1-\alpha}{1-\gamma-\alpha} n^{1-\gamma} \right)^2 \sim (E[S_n])^2.
\end{align*}
\end{proof}

\begin{corollary} \label{cor:190306biasdecayESn2-3} Under the condition of Lemma \ref{lem:190306biasdecayESn2-3}, 
$\dfrac{S_n}{n^{1-\gamma}} \stackrel{L^2}{\to} \dfrac{1-\alpha}{1-\gamma-\alpha}$.
\end{corollary} 

\begin{proof} This follows from
\begin{align*}
E\left[ \left( \dfrac{S_n}{n^{1-\gamma}} - \dfrac{1-\alpha}{1-\gamma-\alpha}\right)^2 \right] &= \dfrac{E[(S_n)^2]}{n^{2(1-\gamma)}} - 2 \cdot \dfrac{1-\alpha}{1-\gamma-\alpha} \cdot \dfrac{E[S_n]}{n^{1-\gamma}} + \left( \dfrac{E[S_n]}{n^{1-\gamma}} \right)^2.
\end{align*}
\end{proof}

\begin{lemma} \label{lem:190306biasdecayESn2-6} If $\alpha \geq 1/2$ and $\gamma = 1-\alpha$, then $E[(S_n)^2] \sim (E[S_n])^2$ as $n \to \infty$.
\end{lemma}

\begin{proof} If $\gamma+\alpha=1$, then Lemma \ref{lem:171113biasedRWESn} implies that
\begin{align*}
\dfrac{\varepsilon_{\ell}  E[S_\ell]}{a'_{\ell+1}} \sim \ell^{-\gamma} \cdot (1-\alpha)  \ell^{\alpha}  \log \ell \left/ \dfrac{\ell^{2\alpha}}{\Gamma(2\alpha+1)} \right. = (1-\alpha)\Gamma(2\alpha+1) \cdot \dfrac{\log \ell}{\ell}
\end{align*}
as $\ell \to \infty$. When $\alpha \geq 1/2$ and $\gamma+\alpha=1$, the second term in the right hand side of \eqref{eq:biasdecayESn2} is dominant as $n \to \infty$, and we have
\begin{align*}
 E[(S_n)^2] &\sim 2\left(1-\alpha\right) a'_n \sum_{\ell=0}^{n-1} \dfrac{\varepsilon_\ell  E[S_\ell]}{a'_{\ell+1}} \\
 &\sim 2\left(1-\alpha\right) \cdot \dfrac{n^{2\alpha}}{\Gamma(2\alpha+1)} \cdot  (1-\alpha)\Gamma(2\alpha+1)\sum_{\ell=1}^{n-1} \dfrac{\log \ell}{\ell} \\
 &\sim \{(1-\alpha) n^{\alpha} \log n\}^2  \sim (E[S_n])^2.
\end{align*}
\end{proof}

By the same argument as in Corollary \ref{cor:190306biasdecayESn2-3}, we have the following. 

\begin{corollary} \label{cor:190306biasdecayESn2-6} Under the same condition as in Lemma \ref{lem:190306biasdecayESn2-6}, $\dfrac{S_n}{n^{\alpha} \log n} \stackrel{L^2}{\to} 1-\alpha$. 
\end{corollary}

This completes the proof of Theorem \ref{thm:ElephantPolynomDecay}.

\appendix

\section{}

\subsection{Lemmas from calculus}

\begin{lemma}[Sch\"{u}tz and Trimper \cite{SchutzTrimper04}, (12) and (13)] \label{SchutzTrimper04Recursion} The general solution $\{x_n\}$ to the recursion 
\[ \begin{cases} \hspace{4.3mm}x_1=f_0, \\
x_{n+1} = f_n+g_n \cdot x_n &[n=1,2,\cdots] \\
\end{cases} \]
is given by
\[ x_n = \sum_{\ell=0}^{n-1} f_\ell \cdot \prod_{k=\ell+1}^{n-1} g_k. \]
\end{lemma}

\begin{lemma}[see e.g. Knopp \cite{Knopp56Dover}, p.34] \label{lem:StolzCesaro} If a real sequence $\{a_n\}$ and a positive sequence $\{b_n\}$ satisfy
\[ \lim_{n\to \infty} \dfrac{a_n}{b_n} = L \in \mathbb{R} \cup \{ \pm \infty \},\quad\mbox{and}\quad \sum_{n=1}^{\infty} b_n = +\infty, \]
then
\[ \lim_{n\to \infty} \dfrac{\sum_{k=1}^n a_k}{\sum_{k=1}^n b_k} = L. \]
\end{lemma}

\begin{lemma} \label{lem:KroneckerHeyde} Consider a positive real sequence $\{a_n\}$ which monotonically diverges to $+\infty$, and another real sequence $\{b_n\}$.
\begin{itemize}
\item[(i)] (Kronecker's lemma) If $\displaystyle \sum_{k=1}^{\infty} \dfrac{b_k}{a_k}$ converges, then $\displaystyle \lim_{n \to \infty} \dfrac{1}{a_n}\sum_{k=1}^n b_k =0$.
\item[(ii)] (Heyde \cite{Heyde77}, Lemma 1 (ii)) If $\displaystyle \sum_{k=1}^{\infty} a_k b_k$ converges, then$\displaystyle \lim_{n \to \infty} a_n\sum_{k=n}^{\infty} b_k =0$.
\end{itemize}
\end{lemma}

\subsection{Martingale limit theorems}

\begin{theorem}[Hall and Heyde \cite{HallHeyde80}, Theorem 2.15] \label{thm:HallHeyde80Thm215} Suppose that $\{M_n\}$ is a square-integrable martingale with mean $0$.
Let $d_k=M_k-M_{k-1}$ for $k=1,2,\cdots$, where $M_0=0$. On the event
\[ \left\{ \sum_{k=1}^{\infty}  E[(d_k)^2 \mid \mathcal{F}_{k-1}] <+\infty \right\},  \]
$\{M_n\}$ converges a.s..
\end{theorem}

The following theorem is a special case of Heyde \cite{Heyde77}, Theorem 1 (b).

\begin{theorem} \label{thm:Heyde77Theorem1b} Suppose that $\{M_n\}$ is a square-integrable martingale with mean $0$. Let $d_k=M_k-M_{k-1}$ for $k=1,2,\cdots$, where $M_0=0$. If  
\[ \displaystyle \sum_{k=1}^{\infty} E[(d_k)^2] < +\infty \]
holds in addition, then we have the following: Let $\displaystyle s_n^2 := \sum_{k=n}^{\infty} E[ (d_k)^2]$.
\begin{itemize}
\item[(i)] The limit $M_{\infty}:=\sum_{k=1}^{\infty} d_k$ exists a.s., and $M_n \stackrel{L^2}{\to} M_{\infty}$.
\item[(ii)] Assume that
\begin{itemize}  
\item[a)] $\displaystyle \dfrac{1}{s_n^2}  \sum_{k=n}^{\infty} (d_k)^2 \to 1$ as $n \to \infty$ in probability, and
\item[b)] $\displaystyle \lim_{n \to \infty} \dfrac{1}{s_n^2}E\left[  \sup_{k \geq n} (d_k)^2\right]=0$.
\end{itemize}
Then we have
\[ \dfrac{M_{\infty} - M_n}{s_{n+1}} = \dfrac{ \sum_{k=n+1}^{\infty} d_k}{s_{n+1}}\stackrel{\text{d}}{\to}  N(0,1). 
\] 
\item[(iii)] Assume that the following three conditions hold:
\begin{itemize}
\item[a')] $\displaystyle \dfrac{1}{s_n^2}  \sum_{k=n}^{\infty} (d_k)^2 \to 1$ as $n \to \infty$ a.s.,
\item[c)] $\displaystyle \sum_{k=1}^{\infty} \dfrac{1}{s_k} E[ |d_k| : |d_k| > \varepsilon s_k] < +\infty$ for any $\varepsilon > 0$, and
\item[d)] $\displaystyle \sum_{k=1}^{\infty} \dfrac{1}{s_k^4} E[ (d_k)^4 : |d_k| \leq \delta s_k] < +\infty$ for some $\delta > 0$.
\end{itemize}
Then $\displaystyle \limsup_{n \to \infty} \pm \dfrac{M_{\infty} - M_n}{\widehat{\phi}(s_{n+1}^2)} =1$ a.s., where $\widehat{\phi}(t):=\sqrt{2t\log |\log t|}$.
\end{itemize}
\end{theorem}

\begin{remark} \label{rem:Heyde77Theorem1b} 
A sufficient condition for b) in Theorem \ref{thm:Heyde77Theorem1b} is that 
\[ \displaystyle \lim_{n \to \infty}\dfrac{1}{s_n^2}  \sum_{k=n}^{\infty} E[ (d_k)^2 : |d_k| > \varepsilon s_n] =0 \]
for any $\varepsilon > 0$. (See the proof of Corollary 1 in Heyde \cite{Heyde77}.)
\end{remark}

\section*{Acknowledgements}

N.K. is partially supported by JSPS Grant-in-Aid for Young Scientists (B) No. 16K17620. M.T. is partially supported by JSPS Grant-in-Aid for Young Scientists (B) No. 16K21039, and JSPS  Grant-in-Aid for Scientific Research (C) No. 19K03514.


\end{document}